\documentclass [11pt,a4paper]{article}
\usepackage[T1]{fontenc}
\usepackage[ansinew]{inputenc}
\usepackage{amssymb}
\usepackage{amsmath}
\usepackage{graphicx}
\usepackage{amsthm}
\usepackage{cite}
\newtheorem{theorem}{\textbf{Theorem}}[section]
\newtheorem{lemma}{\textbf{Lemma}}[section]
\newtheorem{proposition}{\textbf{Proposition}}[section]
\newtheorem{corollary}{\textbf{Corollary}}[section]
\newtheorem{remark}{\textbf{Remark}}[section]
\newtheorem{definition}{\textbf{Definition}}[section]

\allowdisplaybreaks[4]

\def\be{\begin{equation}}
\def\ee{\end{equation}}
\def\bea{\begin{eqnarray}}
\def\eea{\end{eqnarray}}
\def\bt{\begin{theorem}}
\def\et{\end{theorem}}
\def\bl{\begin{lemma}}
\def\el{\end{lemma}}
\def\br{\begin{remark}}
\def\er{\end{remark}}
\def\bp{\begin{proposition}}
\def\ep{\end{proposition}}
\def\bc{\begin{corollary}}
\def\ec{\end{corollary}}
\def\bd{\begin{definition}}
\def\ed{\end{definition}}

\def\s0t{\sup _{0\leq \tau\leq t}}
\def\C0T{C([0,T];\,}

\def\DAS{D( A^{\frac{S}{2}})}
\def\DAS1{D( A^{ \frac{S+1}{2}})}

 \def\non{\nonumber }

\def \no#1#2#3 {{\bf #1} (#3), #2.}
\def \eds#1#2#3 {#1, #2, #3.}

\title{On the rigidity of nematic liquid crystal flow on $\mathbb S^2$}
\date{}

\author{Changyou Wang \footnote{Department of Mathematics,
University of Kentucky, Lexington, KY 40506, USA ({\tt
cywang@ms.uky.edu})}
        \and
        Xiang Xu \footnote{Department of Mathematical Sciences, Carnegie Mellon
University, Pittsburgh, PA 15213, USA ({\tt xuxiang@andrew.cmu.edu})
}
}

\begin{document}

\maketitle

\begin{abstract}
In this paper we establish the uniformity property of a simplified
Ericksen-Leslie system modelling the hydrodynamics of nematic liquid crystals
on the two dimensional unit sphere $\mathbb{S}^2$, namely the uniform convergence
in $L^2$ to a steady state exponentially as $t$ tends to infinity.
The main assumption, similar to Topping \cite{T97}, concerns the equation of liquid crystal
director $d$ and states that at infinity time, a weak limit $d_\infty$ and any bubble
$\omega_i$ ($1\le i\le l$) share a common orientation. As consequences, the uniformity
property holds under various types of small initial data.
\end{abstract}

\section{Introduction}

The  Ericksen-Leslie system modelling the hydrodynamics of nematic liquid crystals was proposed by
Ericksen and Leslie during the period between 1958 and 1968 \cite{Ericksen, Leslie}. It is a macroscopic continuum
description of the time evolution of the materials under the influence of both the flow velocity
field $u(x,t)$ and the macroscopic description of the microscopic orientation configuration
$d(x,t)$ of rod-like liquid crystals (i.e. $d(x,t)$ is a unit vector in $\mathbb R^3$). In order to effectively analyze it, Lin \cite{Lin} proposed a simplified
version of the Ericksen-Leslie system, which is called the
nematic liquid crystal flow and is given by
\begin{equation}
\label{LLF}
\begin{cases}
u_t+u\cdot\nabla u + \nabla P = \Delta u - \nabla d \Delta d,\\
\mbox{div} \ u = 0, \\
d_t + u\cdot\nabla d = \Delta d + |\nabla d|^2d.
\end{cases}
\end{equation}
Roughly speaking, the system (\ref{LLF}) is a coupling between the non-homogeneous
Naiver-Stokes equation and the transported heat flow of harmonic maps to $\mathbb S^2$.
The system (\ref{LLF}) has generated great interests and activities among analysts, and there have been many research works
on (\ref{LLF}) recently (see, e.g. \cite{LL95} \cite{LL96} \cite{LL01} \cite{LLW10} \cite{LW10} \cite{H11} \cite{XZ11}
and \cite{LLZ12}).

In this paper, we are mainly interested in the long time dynamics of the nematic liquid crystal flow
(\ref{LLF})  on $ \mathbb{S}^2 \times (0, +\infty)$, where $\mathbb S^2 \hookrightarrow\mathbb R^3$ is the unit sphere
 that is equipped with the
standard metric $g_0$. To begin with, we introduce some notations and explain terms in (\ref{LLF}).
First, the fluid velocity field $u(x,t)\in T_{x}\mathbb S^2\times \{t\}\equiv T_x\mathbb S^2$,
the liquid crystal molecule director field $d(x,t)\in \mathbb S^2$, and
the pressure function $P(x,t)\in\mathbb R$ for $(x,t)\in \mathbb S^2\times [0,+\infty)$.
In the system (\ref{LLF}), $\nabla$ stands for the gradient operator on $\mathbb S^2$,
$\mbox{div}(={\rm{div}}_{g_0})$ and $\Delta(=\Delta_{g_0}={\rm{div}}_{g_0}\nabla)$ represent the divergence
operator and the Laplace-Beltrami operator on $(\mathbb{S}^2,g_0)$. Hence the second equation (\ref{LLF})$_2$
describes the incompressibility of the
fluid. The convection term $u\cdot\nabla u$ in the first equation (\ref{LLF})$_1$
is the directed differentiation of $u$ with respect to the direction $u$
itself, which is interpreted as the covariant derivative
$D_{u}u$. Here $D$ denotes the covariant
derivative operator on $(\mathbb S^2, g_0)$.  Following the terminology in \cite{TR88},
we know that $\Delta u$ is also accounted for by the Bochner
Laplacian operator $-D^{\ast} D u$, where $D^\ast$ is the adjoint operator of $D$ with
respect to the  $L^2(\mathbb{S}^2, T\mathbb{S}^2)$ inner product.
Write $d=(d^1, d^2, d^3)\in \mathbb{S}^2\hookrightarrow\mathbb{R}^3$.
Then in (\ref{LLF})$_3$, $\displaystyle\nabla d \Delta d=\sum_{i=1}^{3}\nabla d^i \Delta d^i$,
 the $i$-th component of $u\cdot\nabla d$ equals to $g_0(u, \nabla d^i)$,
$\displaystyle\Delta d=\big(\Delta_{g_0} d^1, \Delta_{g_0} d^2, \Delta_{g_0} d^3\big)$,
and $\displaystyle |\nabla d|^2=g_0(\nabla d,\nabla d)$.

We will consider the nematic liquid crystal flow (\ref{LLF}) along with the initial condition:
\begin{equation}\label{IC}
\big(u(x,0), d(x,0)\big)=\big(u_0(x), d_0(x)\big), \ x\in\mathbb S^2,
\end{equation}
where $u_0\in L^2(\mathbb S^2, T\mathbb S^2)$ is a divergence free tangential vector field on ($\mathbb S^2,g_0$) with zero average:
\begin{equation} \int_{\mathbb{S}^2}u_0(x)\,dv_{g_0}=0\  \ {\rm{and}}\ \ {\rm{div}}\ u_0=0 \ {\rm{on}}\ \mathbb S^2,
\label{zero_ave}
\end{equation}
and $\displaystyle d_0\in H^1(\mathbb S^2,\mathbb S^2):=\Big\{d\in H^1(\mathbb S^2,\mathbb R^3):
\ d(x)\in\mathbb R^3 \ {\rm{a.e.}}\ x\in\mathbb S^2\Big\}$.

In the recent paper \cite{LLW10}, Lin, Lin, and Wang have proved the global
existence of weak solutions to the initial and boundary value problem of
the nematic liquid crystal flow (\ref{LLF}) on any bounded domain
$\Omega\subset\mathbb R^2$, that are smooth away from at most
finitely many singular times (see
also \cite{H11} for discussion in $\mathbb{R}^2$). The
uniqueness of the global weak solution constructed by \cite{LLW10}
has been given by
\cite{LW10} (see \cite{XZ11} for a different proof). It is an open and
challenging problem whether there exists a global weak solution to the
problem (\ref{LLF})-(\ref{IC}) in $\mathbb R^3$. On
the other hand, concerning the long time behavior of the global
solution established in \cite{LLW10}, the authors have only shown the
subsequential convergence at time infinity. Since the structure of the set of equilibrium
to the problem (\ref{LLF}) is a
continuum, whether the limiting point of (\ref{LLF}) at time infinity is unique is a very interesting problem.

The aim of this paper is to study this
problem for (\ref{LLF})-(\ref{IC}) under certain assumptions on the initial data. The key
ingredient is to apply a \L ojaciewicz-Simon type approach, which
was originally studied by L. Simon in \cite{S83} for a large class of
nonlinear geometric evolution equations. Due to the nonlinear
constraint $|d|=1$, we cannot establish the related \L
ojaciewicz-Simon type inequality for the Ginzburg-Landau approximation version
of (\ref{LLF}) as in \cite{W10, WXL12}. However,
there is a counterpart for our problem if we consider the domain being the
unit sphere, namely the heat flow of harmonic maps from $\mathbb S^2$ to
$\mathbb S^2$ considered by Topping \cite{T97}.

We would like to point that by slightly modifying the proof of the regularity Theorem 1.2 in \cite{LLW10},  \
the existence Theorem 1.3 of \cite{LLW10} and the uniqueness theorem of \cite{LW10}
can be extended to (\ref{LLF})-(\ref{IC}) on $\mathbb S^2$ without much difficulty.
More precisely, we have
\begin{theorem}\label{existence} Suppose $u_0\in L^2(\mathbb S^2, T\mathbb S^2)$ satisfies (\ref{zero_ave})
and $d_0\in H^1(\mathbb S^2,\mathbb S^2)$. Then there exists a unique, global weak solution
$(u,d)\in \big(L^\infty_tL^2_x(\mathbb S^2\times [0,+\infty), T\mathbb S^2)\big)\times \big(L^\infty_t H^1_x(\mathbb S^2\times
[0,+\infty), \mathbb S^2)\big)$ of (\ref{LLF})-(\ref{IC}) such that
the following properties hold:
\begin{itemize}
\item [(i)] There exists a non-negative integer $L$ depending only on $(u_0,d_0)$ and
$0<T_1<\cdots<T_L<+\infty$ such that
$$(u,d)\in C^\infty\Big(\mathbb S^2\times \big((0,+\infty)\setminus\{T_i\}_{i=1}^L\big)\Big).$$
\item[(ii)] For $1\le i\le L$,
\begin{equation}\label{energy_concen}
\limsup_{t\uparrow T_i}\max_{x\in\mathbb S^2}\int_{\mathbb S^2\cap B_r(x)}(|u|^2+|\nabla d|^2)(y,t)\,dv_{g_0}(y)\ge 8\pi,
\ \forall r>0.
\end{equation}
\item[(iii)] There exist $t_k\uparrow +\infty$, a harmonic map $d_\infty\in C^\infty(\mathbb S^2,\mathbb S^2)$, and nontrivial
harmonic maps $\{\omega_i\}$ ($1\le i\le K$ for some non-negative integer $K$) such that
$$u(t_k)\rightarrow 0 \ {\rm{strongly\ in}}\ H^1(\mathbb S^2), \ d(t_k)\rightarrow d_\infty
\ {\rm{weakly\ in}}\ H^1(\mathbb S^2,\mathbb S^2),$$
and
\begin{equation}\label{energy_id}
\int_{\mathbb S^2}|\nabla d(t_k)|^2\,dv_{g_0}\rightarrow\int_{\mathbb S^2}|\nabla d_\infty|^2\,dv_{g_0}
+\sum_{i=1}^K\int_{\mathbb S^2}|\nabla \omega_i|^2\,dv_{g_0}.
\end{equation}
\item[(iv)] If
$$\int_{\mathbb S^2}(|u_0|^2+|\nabla d_0|^2)\,dv_{g_0}\le 8\pi,$$
then $(u,d)\in C^\infty(\mathbb S^2\times [0,+\infty), T\mathbb S^2\times \mathbb S^2)$.
Moreover, there exist $t_k\uparrow +\infty$ and a harmonic map $d_\infty\in C^\infty(\mathbb S^2,\mathbb S^2)$
such that
$$\big(u(t_k),d(t_k)\big)\rightarrow \big(0,d_\infty\big) \ {\rm{strongly\ in}}\ H^1(\mathbb S^2).$$
\end{itemize}
\end{theorem}

The main results in this paper address the issue of unique limit at time infinity of the solution $(u,d)$ to (\ref{LLF})-(\ref{IC})
given by Theorem \ref{existence}. We find sufficient conditions, similar to that by \cite{T97} on the heat flow of harmonic
maps from $\mathbb S^2$ to $\mathbb S^2$, on $d_\infty$ and $\omega_i,  1\le i\le K$, to guarantee the uniform property
of (\ref{LLF}) at time infinity, i.e., $d_\infty$ is the unique limit of $d$ at $t=+\infty$.
It is worth mentioning that in contrast with the heat flow of harmonic maps considered by \cite{S85} and \cite{T97},
because the local energy inequality of (\ref{LLF}) (see \cite{LLW10} Lemma 4.2) involves
$L^2$-norm of both $|P-c|$ and $|\nabla d|^2$, we cannot show the
uniqueness of the location of bubbling positions of the bubbles $\omega_i$ ($1\le i\le K$).

The paper is written as follows. In \S2, we discuss the uniform limit of (\ref{LLF}) in the space $L^2(\mathbb S^2)$
at $t=+\infty$, and prove the first main theorem 2.1 and corollary 2.1. 
In \S3, we discuss the uniform limit of (\ref{LLF}) in the space $H^k(\mathbb S^2)$ ($k\ge 1$)
at $t=+\infty$, and prove the second main theorem 3.1 and corollary 3.1.


\section{Uniform limit in $L^2(\mathbb S^2)$}
\setcounter{equation}{0}

\noindent Unless explicitly specified, henceforth we will not
distinguish the inner product between $T\mathbb{S}^2$ and
$\mathbb{R}^3$, nor $d \in \mathbb{S}^2 $ and its isometric
embedding $(d^1, d^2, d^3)$ into $\mathbb{R}^3$. For the sake of
simplicity, $\|\cdot\|_{L^2(\mathbb S^2)}$ will be shorthanded by
$\|\cdot\|$, and $\|\cdot\|_{L^p(\mathbb S^2)}$ will be abbreviated
by $\|\cdot\|_{L^p}$ for $p\neq 2$.

We will establish the $L^2$-convergence of the flow $(u,d)$ to (1.1)-(1.3) to a single steady state
solution $(0, d_\infty)$, with $d_\infty\in C^\infty(\mathbb S^2,\mathbb S^2)$ a harmonic map,
as $t$ tends to infinity.  We first recall some
notations introduced by  Topping \cite{T97}.

Let us consider a map $d\in H^1(\mathbb S^2, \mathbb S^2)$.  We use $z = x + iy$ as a complex coordinate
on the domain $\mathbb S^2\equiv\overline{\mathbb C}$, via the stereographic projection.
Set $dz=dx+idy, d\bar{z}=dx-idy$, and write the metric $g_0$ on the domain $\mathbb S^2$
as $\displaystyle \sigma(z)^2 dzd\bar z$, where
$$\sigma(z)=\frac{2}{1+|z|^2}, \ z\in\overline{\mathbb C}$$
Let $u$ denote a complex coordinate on the target $\mathbb S^2$, and write the metric $g_0$ on the target
$\mathbb S^2$ as $\displaystyle \sigma(u)^2 du d\bar u$.
 and write
$$
  d_z = \frac12(d_x-id_y), \ \ d_{\bar{z}}= \frac12(d_x+id_y).$$
The $\partial$-energy and $\bar\partial$-energy of $d$ are given by
 $$\ E_{\partial}(d)=\frac{i}2\int_{\mathbb C}\rho^2(d)|d_z|^2 \,dz\wedge d\bar z,
  \ \ E_{\bar{\partial}}(d)=\frac{i}2\int_{\mathbb C}\rho^2(d)|d_{\bar{z}}|^2\,dz\wedge d\bar z.
$$
Recall the Dirichlet energy of $d$ is defined by
$$E(d):=\frac12\int_{\mathbb S^2}|\nabla d|^2\,dv_{g_0}.$$
It is easy to see that
\begin{align}
   E(d) &= E_{\partial}(d)+E_{\bar{\partial}}(d),
  \label{sum of energy}  \\
  4\pi\,\mbox{deg}(d) &= E_{\partial}(d)-E_{\bar{\partial}}(d).
  \label{difference of energy}
\end{align}
Here ${\rm{deg}}(d)$ denotes the topological degree of $d:\mathbb S^2\to\mathbb S^2$, which is
well-defined for maps $d\in H^1(\mathbb S^2,\mathbb S^2)$ (see Brezis-Nirenberg \cite{BN}).
\begin{theorem}\label{dbar-small} There exist $\epsilon_0>0$ and $T_0\ge 1$ such that if
$(u,d):\mathbb S^2\times (0,+\infty)\to T\mathbb S^2\times \mathbb S^2$ is the global
solution of the nematic liquid crystal flow (1.1)-(1.3) obtained by Theorem \ref{existence}, satisfying
\begin{equation}
\frac12\big\|u(T_0)\big\|^2+2\min\Big\{E_{\partial}(d(T_0)),\ E_{\overline{\partial}}(d(T_0))\Big\}\le\epsilon_0.
\label{dbar-small1}
\end{equation}
then there exist a smooth harmonic map $d_\infty\in C^\infty(\mathbb S^2,\mathbb S^2)$, a nonnegative
integer $k$, and $C_1, C_2>0$  such that\\
(i) as $t\rightarrow +\infty$, it holds that
$$u(t)\rightarrow 0\ {\rm{strongly\ in}}\  H^1(\mathbb S^2),
d(t)\rightarrow d_\infty  \ {\rm{weakly\ in}} \ H^1(\mathbb S^2)
\ {\rm{and\ strongly\ in}}\ L^2(\mathbb S^2).
$$
(ii)
\begin{equation}
\big\|u(t)\|+\big\|d(t)-d_\infty\big\|\le C_1e^{-C_2 t}, \ \forall t\ge T_0.
\end{equation}
(iii)
\begin{equation}
\big|E(d(t))-E(d_\infty)-4\pi k\big|\le C_1e^{-C_2 t}, \ \forall t\ge T_0.
\end{equation}
\end{theorem}

 In order to prove Theorem \ref{dbar-small},  we need a key estimate, originally due to Topping  \cite{T97},
which provides a way to control the $\partial$-energy of $d$ (or $\bar\partial$-energy of $d$)
in terms of its tension field.
\begin{lemma} \label{LS inequality}
There exist $\epsilon_0 > 0$ and $C_0 > 0$ such that if $d\in H^1(
\mathbb{S}^2, \mathbb{S}^2)$ satisfies
\begin{equation}\label{dbar_small_cond}
\min\Big\{E_{\partial}(d), \ E_{\bar\partial}(d)\Big\} <
\epsilon_0.
\end{equation}
Then
\begin{equation}\label{dbar_ineq}
 \min\Big\{E_{\partial}(d), E_{\bar\partial}(d)\Big\}
 \leq C_0\int_{\mathbb S^2} \left| \Delta d + |\nabla d|^2d \right|^2 \,dv_{g_0}.
\end{equation}
\end{lemma}

We need the energy inequality of the solution of (\ref{LLF})-(\ref{IC})
obtained by Theorem \ref{existence}.
\begin{lemma}\label{energy-inequality} Assume that $(u,d):\mathbb S^2\times [0,+\infty)\to T\mathbb S^2\times\mathbb S^2$ is the solution of (\ref{LLF})-(\ref{IC})
obtained by Theorem \ref{existence}. Then for any $t\in (0,+\infty)\setminus\{T_i\}_{i=1}^L$,
the following basic energy law holds:
\begin{align}
\frac{d}{dt}\left[\frac12\int_{\mathbb S^2}  |u|^2 dv_{g_0} + E(d) \right]
=- \int_{\mathbb S^2} \big(\left| \nabla u \right|^2 + \left|
\Delta d + |\nabla d|^2d \right|^2\big) \,dv_{g_0}
 \label{basic energy law}
\end{align}
\end{lemma}
\begin{proof}(\ref{basic energy law}) follows by multiplying both sides of the equation
(\ref{LLF})$_1$ by
$u$ and both sides of the equation (\ref{LLF})$_3$ by $\Delta d+|\nabla d|^2 d$ and integrating the resulting equations over $\mathbb S^2$. We refer the interested readers to \cite{LLW10} Lemma 4.1 for
the detail. 
\end{proof}

We also need the following simple fact on the average of $u$.
\begin{lemma}\label{average} Assume that $(u,d):\mathbb S^2\times [0,+\infty)\to T\mathbb S^2\times\mathbb S^2$ is the solution of (\ref{LLF})-(\ref{IC})
obtained by Theorem \ref{existence}. If the condition (\ref{zero_ave}) holds,
then
\begin{equation}\label{zero_ave1}
\int_{\mathbb S^2} u(x,t)\,dv_{g_0}=0, \ \forall t\ge 0.
\end{equation}
\end{lemma}
\begin{proof} Under the assumption (\ref{zero_ave}), 
(\ref{zero_ave1}) follows by integrating the equation (\ref{LLF})$_1$ over $\mathbb S^2$ and the fact that $\Delta d \nabla d={\rm{div}}\left(\nabla d\otimes\nabla d-\frac12|\nabla d|^2\right)$. Here
$(\nabla d\otimes \nabla d)_{ij}=\langle \frac{\partial d}{\partial x_i}, \frac{\partial d}{\partial x_j}\rangle$, for $1\le i, j\le 3$.
\end{proof}

Now we give the  proof of Theorem \ref{dbar-small}.
\begin{proof} Let $\epsilon_0>0$ be given by Lemma \ref{LS inequality} and
$T_0\ge 1$ be sufficiently  large so that $(u, d)\in C^\infty(\mathbb S^2\times [T_0,+\infty))$.
Without loss of generality, we may assume that
$$E_{\partial}(d(T_0))\le E_{\bar\partial}(d(T_0)).$$
It follows from
\eqref{sum of energy} and \eqref{difference of energy} that
$$
  E_{\partial}(d) = \frac12 \left[ E(d) + 4\pi \mbox{deg}(d)\right].
$$
Since ${\rm{deg}}(d(t))$ is constant for $t\ge T_0$, the basic energy law \eqref{basic energy law} implies
that
\begin{equation} \frac{d}{dt}\Big[\frac12\|u\|^2 + 2E_{\partial}(d) \Big] =
\frac{d}{dt}\Big[\frac12\|u\|^2 + E(d) \Big]=
- \big(\left\| \nabla u \right\|^2 +\left\| \Delta d + |\nabla d|^2d
\right\|^2\big), \ \forall t\ge T_0.
 \label{basic energy law II}
\end{equation}
Therefore, we have that
\begin{equation}\label{dbar_energy_decay}
\frac12\|u(t)\|^2 + 2E_{\partial}(d(t)) \le\frac12\|u(T_0)\|^2 + 2E_{\partial}(d(T_0))\le\epsilon_0, \ \forall t\ge T_0.
\end{equation}
Applying Lemma \ref{dbar-small} to $d(t)$, and Lemma \ref{average} and Poincar\'e's inequality for $u(t)$, we have
\begin{equation}\label{p-bar}
\frac12\left\|u(t)\right\|^2+2E_{\partial}(d(t))\le C[\left\| \nabla u(t) \right\|^2 +\left\| \Delta d(t) + |\nabla d(t)|^2d(t)\|^2\right],
\ \forall t\ge T_0.
\end{equation}
Putting (\ref{p-bar})  together with (\ref{basic energy law II}), we obtain that for all $t\ge T_0$,
\begin{align}
 -\frac{d}{dt}\Big[\frac12\|u\|^2 + 2E_{\partial}(d)
\Big]^{\frac12}
&= \frac{-\frac{d}{dt}\Big[\frac12\| u\|^2 + 2E_{\partial}(d)
\Big]}{2\Big[\frac12\|u\|^2
+ 2E_{\partial}(d) \Big]^{\frac12}}  \non\\
&\geq C\frac{-\frac{d}{dt}\Big[\frac12\|u\|^2 + 2E_{\partial}(d)
\Big]}{\Big[ \|\nabla u\| ^2+ \| \Delta{d} + |\nabla d|^2d
\|^2\Big]^\frac12} \non\\
&\geq C\left[ \|\nabla u\|  + \Big\| \Delta{d} +
|\nabla d|^2d \Big\| \right]  \non\\
&\geq C\left[\frac12\|u\|^2  + 2E_{\partial}(d) \right]^{\frac12}.
  \label{estimate for exponential decay}
\end{align}
Thus by Gronwall's inequality we have
\begin{equation} \Big[\|u(t)\|^2 +
E_{\partial}(d(t))\Big] \leq \Big[\|u(T_0)\|^2+
E_{\partial}(d(T_0))\Big] e^{-C(t-T_0)}\le C\epsilon_0 e^{-C(t-T_0)}, \  \ \forall \ t \geq T_0.
 \label{exponential decay}
\end{equation}
By integrating (\ref{estimate for exponential decay}) over  $[t, +\infty)$, we obtain that
for any $t \geq 2T_0$,
\begin{equation}
\int_{t}^{\infty}  (\big\| \nabla u \big\|+\big\|\Delta d + |\nabla d|^2d
\big\|) \,d\tau \leq C\left[\|u(t)\|^2 + E_{\partial}(d(t))
\right]^{\frac12}\le C_1 e^{-C_2 t}.  \label{L1 convergence}
\end{equation}
Consequently, we infer from (\ref{LLF})$_3$, \eqref{basic
energy law}, \eqref{estimate for exponential decay},
\eqref{exponential decay} and \eqref{L1 convergence} that for any
$t_2 \geq t \geq 2T_0$,
\begin{align} &\|d(t_2)-d(t)\|_{L^1} \leq \int_{t}^{t_2} \|d_t\|_{L^1} d\tau \non\\
 &\leq
\int_{t}^{t_2} \|u\cdot\nabla d\|_{L^1} d\tau + \int_{t}^{t_2}
\left\| \Delta d +
|\nabla d|^2d \right\|_{L^1} d\tau  \non\\
&\leq \int_{t}^{t_2} \|u\|_{L^2}\|\nabla d\|_{L^2} d\tau +
2\sqrt{\pi}\int_{t}^{t_2} \left\| \Delta d +
|\nabla d|^2d \right\|_{L^2} d\tau  \non\\
&\leq C\Big[\int_{t}^{t_2} \|\nabla u\|_{L^2} d\tau +\int_{t}^{t_2} \left\| \Delta d + |\nabla d|^2d
\right\| d\tau\Big]  \non\\
&\leq C\big[\|u(t)\|^2 + E_{\partial}(d(t))
\big]^{\frac12} \leq C_1e^{-C_2t}.
\end{align}
Thus
\begin{equation}
\|d(t_2)-d(t)\|_{L^2}^2 \leq 2\|d(t_2)-d(t)\|_{L^1} \leq
C_1e^{-C_2t}, \label{L2 cauchy sequence}
\end{equation}
which indicates that as $t \rightarrow +\infty$, $d(t)$ converges in
$L^2(\mathbb{S}^2)$. It follows from Theorem \ref{existence} (iii) that there exist a smooth
harmonic map $d_\infty\in C^\infty(\mathbb S^2, \mathbb S^2)$, nontrivial harmonic maps $\{\omega_i\}_{i=1}^l$ for some
nonnegative integer $l$,  and a sequence
$t_i\rightarrow +\infty$ such that
\begin{equation}\label{sub_conv}
\begin{cases}d(t_i)\rightarrow d_\infty \ {\rm{weakly\ in}}\ H^1(\mathbb S^2,\mathbb S^2) \ {\rm{and\ strongly\ in}}
\ L^2(\mathbb S^2,\mathbb S^2),\\
\displaystyle\int_{\mathbb S^2}|\nabla d(t_i)|^2\,dv_{g_0}\rightarrow \int_{\mathbb S^2}|\nabla d_\infty|^2\,dv_{g_0}+\sum_{i=1}^l \int_{\mathbb S^2} |\nabla \omega_i|^2\,dv_{g_0}.
\end{cases}
\end{equation}
Therefore, by choosing $t_2=t_i$ in (\ref{L2 cauchy sequence}) and sending $i$ to $\infty$, we conclude that
\begin{equation}
\|d(t) - d_\infty \|_{L^2}  \le C_1 e^{-C_2 t}, \ \forall t\ge 2T_0.  \label{strong L2 convergence}
\end{equation}
In particular,  (\ref{strong L2 convergence}) implies that $d(t)$ converges to $d_\infty$ in $L^2(\mathbb S^2)$ as $t\rightarrow +\infty$.
To show that $d(t)$ converges to $d_\infty$ weakly in $H^1(\mathbb S^2)$. Let $\{t_j\}$ be any sequence
tending to $+\infty$.  Since $\{d(t_j)\}$ is bounded in $H^1(\mathbb S^2)$, there exists a subsequence
$t_{j'}\rightarrow+\infty$ such that $d(t_{j'})$ weakly converges in $H^1(\mathbb S^2)$ and strongly in
$L^2(\mathbb S^2)$ to a map $d_*\in H^1(\mathbb S^2,\mathbb S^2)$. Hence $d_*=d_\infty$.  This
shows that $d(t)$ converges to $d_\infty$ weakly in $H^1(\mathbb S^2)$ as $t\rightarrow +\infty$. Hence (i) is proven.

To show (iii), integrating (\ref{basic energy law II}) over $[t, t_2]$ for $t\ge 2T_0$ and applying (\ref{exponential decay})
yields
\begin{eqnarray*}
&&\big[\frac12 \|u(t)\|^2+E(d(t))\big]-\big[\frac12 \|u(t_2)\|^2+E(d(t_2))\big]\\
&&=\big[\frac12 \|u(t)\|^2+2E_{\partial}(d(t))\big]-\big[\frac12 \|u(t_2)\|^2+2E_{\partial}(d(t_2))\big]\\
&&\le \big[\frac12 \|u(t)\|^2+2E_{\partial}(d(t))\big]\le C_1e^{-C_2 t}.
\end{eqnarray*}
 This implies
\begin{equation} \label{exp_energy_decay0}
|E(d(t)-E(d(t_2)|\le\big(\|u(t)\|^2+\|u(t_2)\|^2)+C_1e^{-C_2 t}, \ \forall t\ge 2T_0.
\end{equation}
Let $t_2\rightarrow+\infty$ be the sequence such that (\ref{sub_conv}) holds. Since each harmonic map
$\omega_i$, $1\le i\le l$, is nontrivial and has its energy
$$\int_{\mathbb S^2}|\nabla\omega_i|^2\,dv_{g_0}=8\pi m_i,$$
for some positive integer $m_i$, there exists a nonnegative integer $k$ such that
$$\lim_{t_2\rightarrow +\infty} E(d(t_2))=E(d_\infty)+8\pi k.$$
Sending $t_2$ to infinity in (\ref{exp_energy_decay0}), this implies
\begin{equation} \label{exp_energy_decay1}
\Big|E(d(t)-E(d_\infty)-8\pi k\Big|\le C_1e^{-C_2 t}, \ \forall t\ge 2T_0.
\end{equation}
The proof is now complete.
\end{proof}

It is well-known that any harmonic map from $\mathbb S^2$ to $\mathbb S^2$ is either holomorphic or anti-holomorphic.
Inspired by \cite{T97} Theorem 2, we have

\begin{corollary} Suppose $(u,d):\mathbb S^2\times (0,+\infty)\to T\mathbb S^2\times \mathbb S^2$ is the global
solution of the nematic liquid crystal flow (1.1)-(1.3) obtained by Theorem \ref{existence}. \\
(i) Suppose that the weak limit $d_\infty$ and the bubbles
$\omega_i$ ($1\le i \le l$), associated with a sequence $t_i\uparrow +\infty$, are all holomorphic or all
anti-holomorphic. Then  there exist a nonnegative
integer $k$, and $C_1, C_2>0$  such that\\
$$u(t)\rightarrow 0\ {\rm{strongly\ in}}\  H^1(\mathbb S^2),
d(t)\rightarrow d_\infty  \ {\rm{weakly\ in}} \ H^1(\mathbb S^2)
\ {\rm{and\ strongly\ in}}\ H^1(\mathbb S^2), \ {\rm{as}}\ t\rightarrow +\infty,$$
and
\begin{equation}
\big\|u(t)\|+\big\|d(t)-d_\infty\big\|+\big|E(d(t))-E(d_\infty)-4\pi k\big|\le C_1e^{-C_2 t}, \ \forall t\ge T_0.
\end{equation}
(ii) The same conclusions of (i) hold if the initial data $(u_0,d_0)$ satisfies
\begin{equation}\label{dbar-small2}
\frac12\|u_0\|^2+2\min\Big\{E_\partial(d_0), \ E_{\bar\partial}(d_0)\Big\}<8\pi.
\end{equation}
\end{corollary}
\begin{proof} For the part (i), it suffices to verify that the condition (\ref{dbar-small1}) holds.
For simplicity, assume that $d_\infty$ and all $\omega_i$'s are anti-holomorphic. Thus we have
$$E_\partial(d_\infty)=E_\partial(\omega_1)=\cdots=E_\partial(\omega_l)=0.$$
From (1.5), we know that
$$\lim_{t_i\uparrow+\infty}\big[\frac12\|u(t_i)\|^2+E_\partial(d(t_i))\big]=E_\partial(d_\infty)+\sum_{i=1}^lE_\partial(\omega_i)=0.$$
This clearly implies that there exists a sufficiently large $i_0$ such that
$$\big[\frac12\|u(t_{i_0})\|^2+2E_\partial(d(t_{i_0}))\big]\le \epsilon_0,$$
which implies (\ref{dbar-small1}).

For the part (ii), we will show that (\ref{dbar-small2}) implies the condition in the part (i).
For simplicity, assume that $\displaystyle E_\partial(d_0)\le E_{\bar\partial}(d_0)$. Let $0<T_1<+\infty$ be the first singular time of the flow
(1.1)-(1.3).  Since $(u(t), d(t))\in C^\infty(\mathbb S^2\times (0,T_1))$ and
$$\lim_{t\downarrow 0}\big[\|u(t)-u_0\|+\|\nabla(d(t)-d_0)\|\big]=0,$$
it is not hard to see
$${\rm{deg}}(d(t))={\rm{deg}}(d_0), \ \forall \ 0\le t<T_1.$$
Thus the basic energy law \eqref{basic energy law} implies
that
\begin{equation} \frac{d}{dt}\Big[\frac12\|u\|^2 + 2E_{\partial}(d) \Big] =
\frac{d}{dt}\Big[\frac12\|u\|^2 + E(d) \Big]\le 0. \label{2.22}
\end{equation}
Integrating (\ref{2.22}) from $0$ to $0<t\le T_1$ yields
\begin{equation}\label{mono1}\frac12\|u(t)\|^2+2E_{\partial}(d(t))\le \frac12\|u_0\|^2+2E_{\partial}(d_0), \ \ 0\le t\le T_1.
\end{equation}
Let $T_2\in (T_1, +\infty)$ be the second singular time.  Then the same argument yields
\begin{equation}\label{mono2}
\frac12\|u(t)\|^2+2E_{\partial}(d(t))\le \frac12\|u(T_1)\|^2+2E_{\partial}(d(T_1)), \ \ T_1\le t\le T_2.
\end{equation}
Since there are at most finitely many finite singular times $\{T_i\}_{i=1}^L$ for the flow (1.1)-(1.3), by repeating the argument
we would reach that for any $t\ge 0$, it holds that
\begin{eqnarray}\label{mono3}
\mathcal E(t):=\frac12\|u(t)\|^2+2E_\partial(d(t)))&\le &
\frac12\|u_0\|^2+2E_\partial(d_0)\non\\
&=&\frac12\|u_0\|^2+2\min\Big\{E_\partial(d_0),
E_{\bar\partial}(d_0)\Big\}\non\\
&<&8\pi.
\end{eqnarray}
By the lower semicontinuity, we have that
$$2E_{\partial}(d_\infty)\le \lim_{t\rightarrow\infty} \mathcal E(t)<8\pi,$$
and
$$2E_{\partial}(\omega_i)\le \lim_{t\rightarrow\infty} \mathcal E(t)<8\pi,\ 1\le i\le l.$$
This implies that $\omega_1,\cdots, \omega_l$ are all nontrivial anti-holomorphic maps. If $d_\infty$ is not a constant, then $d_\infty$
has to be anti-holomorphic. Therefore $d_\infty$ and all $\omega_i$'s are anti-holomorphic.

Thus the conclusions in (i) and (ii) follow from Theorem \ref{dbar-small}. The proof is complete.
\end{proof}

\section{Uniform limit in $H^k(\mathbb S^2)$ for $k\ge 1$}
\setcounter{equation}{1}

This subsection is to consider the convergence issues of the nematic liquid crystal flow (1.1)-(1.3) in higher order Sobolev spaces
at $t=+\infty$.

\begin{theorem}\label{hk-convergence} Suppose $(u,d):\mathbb S^2\times [0,+\infty)\to T\mathbb S^2\times\mathbb S^2$ is the
global solution of (1.1)-(1.3) obtained by Theorem \ref{existence}.
Suppose that there exist a sequence $t_i\uparrow +\infty$ and a smooth harmonic map
$d_\infty\in C^\infty(\mathbb S^2,\mathbb S^2)$ such that
\begin{equation}\label{subsequence_conv}
\lim_{t_i\uparrow +\infty} \Big[\|u(t_i)\|+\|d(t_i)-d_\infty\|+\|\nabla(d(t_i)-d_\infty)\|\Big]=0.
\end{equation}
Then for any $k\ge 1$ there exist $C_1, C_2>0$ depending only on $k$  such that
\begin{equation}
\Big\|u(t)\Big\|_{H^k(\mathbb S^2)}+\Big\|d(t)-d_\infty\Big\|_{H^k(\mathbb S^2)}\le C_1 e^{-C_2 t}. \label{uniform_conv0}
\end{equation}
In particular, for any $k\ge 1$, $d(t)\rightarrow d_\infty$ in $H^k(\mathbb S^2)$ as $t\rightarrow +\infty$.
\end{theorem}
\begin{proof} For simplicity, assume that $d_\infty$ is anti-holomorphic, i.e., $\partial_z d_\infty\equiv 0$. Thus we have
$$E_{\partial}(d(t_i))=E_{\partial}(d(t_i)-d_\infty)\le E(d(t_i)-d_\infty).$$
This, combined with (\ref{subsequence_conv}), implies
\begin{equation}\label{dbar-small4}
\lim_{t_i\uparrow +\infty}\Big[\|u(t_i)\|+E_{\partial}(d(t_i))\Big]=0.
\end{equation}
Hence we can apply Theorem \ref{dbar-small} to conclude that
$$u(t)\rightarrow 0 \ {\rm{strongly\ in}}\ H^1(\mathbb S^2),
d(t)\rightarrow d_\infty \ {\rm{weakly\ in}}\ H^1(\mathbb S^2) \ {\rm{and\ strongly\ in }}\ L^2(\mathbb S^2),
\ {\rm{as}}\ t\rightarrow\infty,$$
and
\begin{equation}\label{exp_dec}
\|u(t)\|+\|d(t)-d_\infty\|\le C_1e^{-C_2 t}.
\end{equation}
Since it follows from the basic energy law (\ref{basic energy law}) and
$$\displaystyle \lim_{t_i\uparrow+\infty}
\big(\frac12\|u(t_i)\|^2+E(d(t_i))\big)=E(d_\infty)$$
that
$$\lim_{t\rightarrow +\infty} E(d(t))=E(d_\infty),$$
we can conclude that
\begin{equation}\lim_{t\rightarrow +\infty}\Big\|d(t)-d_\infty\Big\|_{H^1(\mathbb S^2)}=0. \label{h1-conv}
\end{equation}
For $\epsilon_1>0$, let $r_0=r_0(\epsilon_1)>0$ be such that
$$\max_{x\in\mathbb S^2} \int_{B_{r_0}(x)\cap\mathbb S^2}|\nabla d_\infty|^2\,dv_{g_0}\le \frac{\epsilon_1}2.$$
By (\ref{h1-conv}), there exists $T_0>0$ such that
\begin{equation}
\sup_{t\ge T} \max_{x\in\mathbb S^2} \int_{B_{r_0}(x)\cap\mathbb S^2}|\nabla d(t)|^2\,dv_{g_0}\le \epsilon_1.
\label{uniform_small}
\end{equation}
As in \cite{LLW10}, (\ref{uniform_small}) then implies the following inequality:
\begin{equation}\label{h1-bound}
\int_{T_0}^\infty\int_{\mathbb S^2}\left(|\nabla u|^2+|\nabla^2 d|^2\right)\le C\left(\|u(T_0)\|^2+E(d(T_0))\right).
\end{equation}
With the estimates (\ref{uniform_small}) and (\ref{h1-bound}), we can apply the regularity Theorem 1.2 of \cite{LLW10}
to get that for any $k\ge 0$, there exists $C_k>0$ such that
\begin{equation}
\sup_{t\ge T_0} \Big(\|u(t)\|_{C^k(\mathbb S^2)}+\|d(t)\|_{C^{k+1}(\mathbb S^2)}\Big)\le C_k. \label{ck-bound}
\end{equation}
By standard interpolation inequalities, (\ref{exp_dec}) and (\ref{ck-bound}) imply that
(\ref{uniform_conv0}) holds. The proof is now complete.
\end{proof}

It is an interesting question to find sufficient conditions that guarantee
the global solution $(u,d):\mathbb S^2\times [0,+\infty)\to\mathbb S^2$ of the flow
(1.1)-(1.3) by Theorem \ref{existence} has a sequence $t_i\uparrow +\infty$ such that $(u(t_i), d(t_i))\rightarrow (0, d_\infty)$
strongly in $L^2(\mathbb S^2)\times H^1(\mathbb S^2)$.

In this context, we have the following result.
\begin{corollary}\label{uniform_conv2}
 Suppose $(u,d):\mathbb S^2\times [0,+\infty)\to T\mathbb S^2\times\mathbb S^2$ is the
global solution of (1.1)-(1.3) by Theorem \ref{existence}. Then there exists a smooth harmonic map
$d_\infty\in C^\infty(\mathbb S^2,\mathbb S^2)$ such that for any $k\ge 1$,
\begin{equation}
\Big\|u(t)\Big\|_{H^k(\mathbb S^2)}+\Big\|d(t)-d_\infty\Big\|_{H^k(\mathbb S^2)}\le C_k e^{-C_k t}, \label{uniform_conv3}
\end{equation}
provided that one of the following conditions holds:\\
i) $\displaystyle \frac12\|u_0\|^2+E(d_0)\le 4\pi$.\\
ii) $d_0(\mathbb S^2)$ is contained in the hemisphere (e.g. $d_0^3\ge 0$).\\
iii)  there exists $C_L>0$ such that $(u_0,d_0)$ satisfies
\begin{equation}
\exp\big({108C_{_L}^8\big(\|u_0\|^2+\frac{1}{8C_{_L}^4}\big)^2}\big)\|\nabla{d}_0\|^2\le\frac{1}{8C_{_L}^4}.
\label{small 1}
\end{equation}
\end{corollary}
\begin{proof} We will establish that under any one of the three conditions, there exists a sequence $t_i\uparrow +\infty$
such that $u(t_i)\rightarrow 0$ in $H^1(\mathbb S^2)$ and $d(t_i)$ is strongly
convergent in $H^1(\mathbb S^2)$.

Let us first consider the condition (i).  It has been proved by Theorem \ref{existence} (iv) that $(u,d)$ has neither finite time
singularity nor energy concentration at $t=\infty$. In particular,
$(u,d)\in C^\infty(\mathbb S^2\times (0,+\infty))$,  and there exists $t_i\uparrow +\infty$ and a harmonic map
$d_\infty\in C^\infty(\mathbb S^2,\mathbb S^2)$ such that
$$\|u(t_i)\|+\|d(t_i)-d_\infty\|_{H^1(\mathbb S^2)}\rightarrow 0.$$
Moreover, it has been shown in \cite{LLW10} that $d_\infty$ is constant, unless
$\displaystyle\big(\frac12\|u_0\|^2+E(d_0)\big)=4\pi$ which would imply that
$u\equiv u_0\equiv 0$, and $d\equiv d_0\equiv d_\infty$ is a harmonic map of degree one.

Now let us consider the condition (ii).  Since $d_0^3\ge 0$, it follows from the maximum principle on the equation (\ref{LLF})$_3$
that $d^3(t)\ge 0$ for all $t\ge 0$. Since there doesn't exist non-constant harmonic maps from $\mathbb S^2$ to
the hemisphere $\mathbb S^2_+$, we can apply \cite{LLW10} Theorem 1.3  to conclude that
there is  neither finite time singularity nor any energy concentration at $t=\infty$ for $(u,d)$
(see also \cite{LLZ12} for a different proof). In particular, there exist a point $p\in \mathbb S^2_+$ and $t_i\uparrow +\infty$
such that
$$\|u(t_i)\|+\|d(t_i)-p\|_{H^1(\mathbb S^2)}\rightarrow 0.$$

Finally let us consider the condition (iii). Recall that under the condition (iii),  it has been proven by Xu-Zhang \cite{XZ11}
that $(u, d)$ is smooth when the domain is $\mathbb R^2$.  Here we indicate how to extend the argument by \cite{XZ11} to $\mathbb S^2$.

Multiplying both sides of the equation (\ref{LLF})$_3$ by $-(\Delta
d+|\nabla d|^2d)$ and integrating the resulting equation over $\mathbb S^2$ yields
\begin{equation}\label{eq:2.4} \frac12\frac{d}{dt}\|\nabla d\|^2+\int_{\mathbb S^2}
\left| \Delta d+|\nabla d|^2d\right|^2 \,dv_{g_0} =
\int_{\mathbb S^2}u\cdot\nabla d \Delta d \, dv_{g_0}.
\end{equation}
To estimate the right side of (\ref{eq:2.4}), 
we apply the Ricci identity on $(\mathbb S^2, g_0)$ and the Poincar\'e inequality for $u$ and $\nabla d$
to obtain
\begin{align}
&\int_{\mathbb S^2}|\Delta d|^2 \,dv_{g_0} = \int_{\mathbb S^2}
(|\nabla^2d|^2+|\nabla d|^2)\,dv_{g_0}\ge \int_{\mathbb S^2}|\nabla^2 d|^2\,dv_{g_0}, \label{ricci}\\
&\|u\|\leq C_{\rm{p}}\|\nabla{u}\|\
\ {\rm{and}}\ \ \|\nabla{d}\|\leq C_{\rm{p}}\|\nabla^2d\| \le C_{\rm{p}}\|\Delta{d}\|, \label{poincare}
\end{align}
where $C_{\rm{p}}>0$ is the constant in the Poincar\'{e} inequality.
Also recall the Ladyzhenskaya inequality on $\mathbb S^2$:
$$ \|f\|_{L^4}\leq C_1\|f\|^\frac12\|\nabla{f}\|^\frac12+C_2\|f\|.$$
Combining these inequalities with Young's inequality and \eqref{basic energy law}, we have
\begin{align}
\label{eq:2.5} &\big|\int_{\mathbb S^2}u\cdot\nabla d \Delta d
\,dv_g\big|\leq \|\Delta
d\|\|\nabla d\|_{L^4}\|u\|_{L^4}  \nonumber\\
&\leq \|\Delta
d\|\big(C_1\|\nabla{d}\|^\frac12\|\Delta{d}\|^\frac12+C_2\|\nabla{d}\|
\big)\big(C_1\|u\|^{\frac12}\|\nabla u\|^{\frac12}+C_2\|u\|\big) \nonumber\\
&\leq \|\Delta
d\|\Big(C_1\|\nabla{d}\|^\frac12\|\Delta{d}\|^\frac12+C_2C_{\rm{p}}^\frac12\|\nabla{d}\|^\frac12\|\Delta{d}\|^\frac12
\Big)\Big(C_1\|u\|^{\frac12}\|\nabla u\|^{\frac12}+C_2C_{\rm{p}}^\frac12\|u\|^{\frac12}\|\nabla u\|^{\frac12} \Big) \nonumber\\
&\leq C_L^2\|\Delta{d}\|^\frac32\|\nabla{d}\|^\frac12\|u\|^\frac12\|\nabla{u}\|^\frac12 \nonumber\\
 &\leq \frac{\|\Delta d\|^2}{8} +
54C_L^8\|u\|^2\|\nabla{u}\|^2\|\nabla d \|^2 \nonumber\\
&\leq \frac{\|\Delta d\|^2}{8} + 54C_L^8\left( \|u_0\|^2+\|\nabla
d_0\|^2\right)\|\nabla u\|^2\|\nabla d\|^2.
\end{align}
 Here
\begin{equation}
 C_L\doteq\big(C_1+C_2C_{\rm{p}}^\frac12 \big).
 \label{def of constant}
\end{equation}
On the other hand, since
$$\big|\Delta d+|\nabla d|^2d\big|^2=|\Delta d|^2-|\nabla d|^4,$$
we have
\begin{equation}\label{eq:2.6}
\int_{\mathbb S^2} \left| |\nabla d|^2d+\Delta d \right|^2 \,dv_g
=\|\Delta d\|^2 - \|\nabla d\|_{L^4}^4 \geq \|\Delta
d\|^2- C_L^4\|\nabla d\|^2\|\Delta d\|^2.
\end{equation}
If $d_0$ satisfies $\|\nabla d_0\|^2<\frac{1}{8C_L^4}, $ then there
exists $T_1
>0$ such that for any  $t\in [0, T_1]$,
\begin{equation}\label{eq:2.7}
\|\nabla d(t)\|^2\leq \frac{1}{8C_L^4}.
\end{equation}
Let $T_1^*$ denote the maximal time such that (\ref{eq:2.7}) holds
on $[0,T_1^*]$. Then, by \eqref{eq:2.4}-\eqref{eq:2.6} we have
that for any $t \in [0, T^\ast_1]$,
\begin{equation}\label{eq:2.8} \frac{d}{dt}\|\nabla d\|^2+\frac14{\|\Delta
d\|^2}\leq 108C_L^8\Big( \|u_0\|^2+\frac{1}{8C_L^4}
\Big)\|\nabla u\|^2\|\nabla d\|^2.
\end{equation}
Using Gronwall's inequality and \eqref{basic energy law}, we deduce
from \eqref{eq:2.8} that for any $0\le t\le T_1^*$,
 \begin{align} \|\nabla d(t)\|^2+\frac14\int_{0}^{t}
\|\Delta d(\tau)\|^2d\tau &\leq \exp\left({108C_L^8\big(
\|u_0\|^2+\frac{1}{8C_L^4} \big) \int_{0}^{T^\ast_1}\|\nabla
u(\tau)\|^2d\tau }\right)\|\nabla d_0\|^2\nonumber\\
&\leq \exp\left({108C_L^8\big( \|u_0\|^2+\frac{1}{8C_L^4} \big)^2}\right)\|\nabla d_0\|^2,
\end{align}
which implies that $T_1^*=T$ and
\begin{equation}
\|\nabla d(t)\|^2+\frac{1}{4}\int_{0}^{t} \|\Delta d(\tau)\|^2d\tau\le
\frac{1}{8C_L^4},
 \label{L2 integrability}
 \end{equation}
holds for all $0\le t\le T$, provided that ($u_0, d_0$) satisfies
\begin{equation} \exp\left({108C_L^8\big(
\|u_0\|^2+\frac{1}{8C_L^4} \big)^2}\right)\|\nabla d_0\|^2\le \frac{1}{8C_L^4}.
\end{equation}
Let $T^*$ be the maximal existence time for the solution $(u,d)$.
Then  \eqref{basic energy law} and \eqref{L2
integrability} ensure that $T^*=+\infty$ by the continuity argument .
Moreover,  $(u,d)\in C^\infty(\mathbb S^2\times(0,+\infty))$ by Theorem 1.2 in \cite{LLW10}.
Note that \eqref{L2 integrability} and (\ref{ricci}) imply that
\begin{equation}
\int_0^\infty \|\nabla^2{d}(\tau)\|^2d\tau \le \frac{1}{2C_L^4}.
\label{L2 integrability for d}
\end{equation}
It follows from (\ref{basic energy law}), (\ref{L2 integrability for d}), and (\ref{poincare})
that there is a sequence $t_i\uparrow +\infty$ such that
$$\lim_{t_i\uparrow +\infty}\Big(\|u(t_i)\|+\|\nabla d(t_i)\|+\|\nabla^2 d(t_i)\|\Big)=0.$$
In particular, there exists $p\in\mathbb S^2$ such that $d(t_i)\rightarrow p$ strongly in
$H^1(\mathbb S^2)$ as $t_i\uparrow +\infty$.

We have verified the condition of Theorem \ref{hk-convergence} holds under all the three
conditions. Thus the conclusion follows  from Theorem \ref{hk-convergence}.
 \end{proof}

\section*{Acknowledgments}  C. Y. Wang is partially supported by
NSF grants 1001115 and 1265574, NSF of China grant 11128102, and a Simons Fellowship in Mathematics.
X. Xu was partially supported by the NSF grant
DMS-0806703. X. Xu is grateful to Professor Y.X. Dong for helpful
discussions. X. Xu would also warmly thank the Center for Nonlinear
Analysis at Carnegie Mellon University (NSF Grants No. DMS- 0405343
and DMS-0635983), where part of this research was carried out.


\end{document}